\documentclass[12pt,a4paper]{amsart}
\usepackage{amsmath}
\usepackage{amsthm}
\usepackage{amssymb}
\usepackage{drpack}
\usepackage[english]{babel}
\usepackage{verbatim}
\usepackage[shortlabels]{enumitem}

\newtheoremstyle{mio}%
{}{} 
{\itshape}{} 
{\bfseries}{.}{ } 
{#1 #2\thmnote{~\mdseries(#3)}} 
\theoremstyle{mio}
\newtheorem{teor}{Theorem}[section]

\newtheorem{prop}[teor]{Proposition}

\newtheoremstyle{definition2}%
{}{} 
{}{} 
{\bfseries}{.}{ } 
{#1 #2\thmnote{\mdseries~ #3}} 
\theoremstyle{definition2}
\newtheorem{ex}[teor]{Example}

\title{The reciprocal complement of a curve}
\author{Dario Spirito}
\address{Dipartimento di Scienze Matematiche, Fisiche e Informatiche, Universit\`a di Udine, Udine, Italy}
\email{dario.spirito@uniud.it}

\newcommand{\recip}{\mathcal{R}}
\newcommand{\local}{\mathcal{O}}
\newcommand{\quot}{\mathcal{Q}}
\newcommand{\mm}{\mathfrak{m}}
\newcommand{\aff}{\mathbb{A}}
\newcommand{\proj}{\mathbb{P}}

\begin{document}
\begin{abstract}
We give a geometric interpretation of the reciprocal complement of an integral domain $D$ in the case $D$ is a one-dimensional finitely generated algebra over an algebraically closed field.
\end{abstract}

\maketitle
\section{Introduction}
Let $D$ be an integral domain. The \emph{reciprocal complement} of $D$ is the ring $\recip(D)$ generated by the inverses of the nonzero elements of $D$ \cite{GLO-egyptian,epstein-egypt-23b}; this notion arises in the context of studying whether a domain is \emph{Egyptian}, i.e., if it is the case that every element of the domain can be written as a sum of reciprocals of elements of the domain. In particular, a domain is Egyptian if and only if $\recip(D)$ is the whole quotient field of $D$.

While $\recip(D)$ is always a local ring, its structure is often complicated and still mostly poorly understood. For example, if $D=k[x_1,\ldots,x_n]$ is a polynomial ring in $n>1$ indeterminates, the ring $\recip(D)$ is an $n$-dimensional domain that is non-Noetherian, non-integrally closed, with infinitely many primes of height $i$ for each $0<i<n$, but nevertheless atomic \cite{EGL-RD-polinomyal}.

In this note, we consider integral domains $D$ that are finitely generated one-dimensional algebras over an algebraically closed field $k$, and connect the study of $\recip(D)$ with the geometric properties of the affine curves $X$ whose ring of regular functions is $D$. We show that we can discriminate whether $\recip(D)$ is a field or not on the basis of the points of $\overline{X}\setminus X$, where $\overline{X}$ is the projective closure of $X$; moreover, when $\recip(D)$ is not a field and $X$ is nonsingular, we connect its properties with the Weierstrass semigroup of $\overline{X}$ at the point of infinity. In particular, we show that, for Dedekind domains, the condition of $\recip(D)$ being a DVR is equivalent to $X$ having genus $0$, giving a geometric interpretation of \cite[Theorem 6.1]{guerrieri-recip}.

\bigskip

Throughout the paper, $k$ is an algebraically closed field, and a \emph{curve} is an irreducible variety of dimension $1$. If $X\subseteq\aff^n_k$ is an affine variety, we denote by $\overline{X}$ its projective closure (i.e., its closure in $\proj^n_k$). If $D$ is a finitely generated $k$-algebra, a \emph{realization} of $D$ is an affine variety $X\subset\aff^n_k$ such that $k[X]\simeq D$; we say that $X$ is \emph{regular at infinity} if the points of $\overline{X}\setminus X$ are regular in $\overline{X}$. Note that two realizations $X,X'$ of $D$ are isomorphic, but $\overline{X}$ and $\overline{X'}$ may not be: for example, $\aff^1_k$ and $V(x-y^3)\subset\aff^2_k$ are both realizations of $k[x]$, but $\overline{X}=\proj^1$ is regular while $\overline{X'}$ is not. We say that an affine curve $X$ is \emph{Egyptian} if $k[X]$ is an Egyptian domain.

If $D$ is an integral domain, we denote by $\quot(D)$ its quotient field; if $D$ is local, we denote by $\mm_D$ its maximal ideal. A \emph{valuation ring} is an integral domain $V$ such that, for every $f\in\quot(V)$, at least one of $f$ and $1/f$ is in $V$; a valuation domain is always local. A \emph{discrete valuation ring} (DVR) is a Noetherian valuation domain; a DVR is either a field or one-dimensional. If $V$ is a DVR, there is a function $\mathbf{v}:\quot(V)\setminus\{0\}\longrightarrow\insZ$  such that, for every $a,b\in\quot(D)$, we have $\mathbf{v}(ab)=\mathbf{v}(a)+\mathbf{v}(b)$ and $\mathbf{v}(a+b)\geq\min\{\mathbf{v}(a),\mathbf{v}(b)\}$ (provided $a+b\neq 0)$. If $D$ is a domain and $K$ a field containing $D$, we denote by $\Zar(K|D)$ the set of all valuation rings containing $D$ and having $K$ as its quotient field. If $\mathfrak{p}$ a prime ideal of $D$, there is always a $V\in\Zar(K|D)$ dominating $D_\mathfrak{p}$; in particular, $D$ is a field if and only if $\Zar(\quot(D)|D)=\{\quot(D)\}$. The intersection of all valuation domains in $\Zar(K|D)$ is equal to the integral closure of $D$ in $K$. If $K,L$ are fields with $\mathrm{trdeg}(L/K)=1$, then every element of $\Zar(L|K)$ is either $L$ or a one-dimensional DVR. See \cite[Chapter 5]{atiyah} and \cite[Chapter 3]{gilmer} for further properties of valuation domains.

\section{Results}
The main theorem of this note is the following.
\begin{teor}\label{teor:main}
Let $k$ be an algebraically closed field, and let $D$ be a finitely generated integral one-dimensional $k$-algebra. Then, the following are equivalent:
\begin{enumerate}[(i)]
\item\label{teor:main:RD} $\recip(D)\neq\quot(D)$ (i.e., $D$ and $X$ are not Egyptian);
\item\label{teor:main:Xn} if $X$ is a realization of $D$ that is regular at infinity, then $|\overline{X}\setminus X|=1$;
\item\label{teor:main:norm} if $X$ is a realization of $D$ and $\nu:Y\longrightarrow \overline{X}$ is a normalization of $\overline{X}$, then $|\nu^{-1}(\overline{X}\setminus X)|=1$.
\end{enumerate}
Moreover, if these conditions hold and $\overline{X}\setminus X=\{p\}$, then the integral closure of $\recip(D)$ is the local ring $\local_{\overline{X},p}$.
\end{teor}
\begin{proof}
We first show that the first two conditions are equivalent. Let $X$ be a realization of $D$ that is regular at infinity; we distinguish two cases.

Suppose that $\overline{X}\setminus X=\{p\}$ is a single point, and let $f\in D$. Then, $f\in K[X]$ is a rational function over $\overline{X}$; we claim that $1/f\in\local_{\overline{X},p}$. If $f$ is constant this is trivial. If $f$ is not constant, then it can't be regular on the whole $\overline{X}$, and thus $f$ is not regular at $p$, and in particular $f\notin\local_{\overline{X},p}$. Since $X$ is regular at infinity, $p$ is a nonsingular point of $\overline{X}$, and thus $\local_{\overline{X},p}$ is a discrete valuation ring; therefore, $1/f\in\local_{\overline{X},p}$. It follows that $\recip(D)\subseteq\local_{\overline{X},p}$, and thus $\recip(D)\neq\quot(D)$.

Suppose that $|\overline{X}\setminus X|>1$. We can identify $K[X]$ with $D$: we claim that $\recip(D)$ is not contained in any $V\in\Zar(\quot(D)|k)\setminus\{\quot(D)\}$. Fix thus any such $V$: then, $V$ dominates a local ring $\local_{\overline{X},p}$ for some $p\in\overline{X}$ \cite[Chapter 7, Theorem 1]{fulton}. If $p\in X$, we can find $f\in K[X]\cap\mm_{\local_{X,p}}\subseteq V$ (just take a regular function on $X$ with a zero at $p$); then, $1/f\in\recip(D)\setminus\mm_V$. If $p\notin X$, choose a different point $q\in\overline{X}\setminus X$ (which exists since $|\overline{X}\setminus X|>1$), and let $Z:=\overline{X}\setminus\{q\}$. Then, $Z$ is affine \cite[Proposition 5]{goodman-affineopen}, and thus there is a nonconstant function $f\in K[Z]$ such that $f(p)=0$; since $X\subseteq Z$, we have $f\in K[X]=D$. Hence, $f\in\mm_{\local_{X,p}}$, and thus $1/f\notin\local_{\overline{X},p}$; it follows that $\recip(D)\nsubseteq\local_{\overline{X},p}$. Since $\recip(D)$ is not contained in any $V\in\Zar(\quot(D)|k)\setminus\{\quot(D)\}$, we must have $\recip(D)=\quot(D)$.

\bigskip

Let now $X$ be any realization of $D$. By blowing up repeatedly the points in $\overline{X}\setminus X$, we can find a projective curve $Z$ and a map $\phi:Z\longrightarrow\overline{X}$ such that $|\phi^{-1}(p)|=1$ for all $p\in X$, $\phi|_{\phi^{-1}(X)}$ is an isomorphism from $Z_0:=\phi^{-1}(X)$ to $X$ and every point of $Z\setminus Z_0=\phi^{-1}(\overline{X}\setminus X)$ is regular.

Then, $Z_0$ is a realization of $D$ (since $k[Z_0]\simeq k[X]\simeq D$), and is regular at infinity by construction, since $\overline{Z_0}=Z$. By the previous part of the proof, $\recip(D)\neq\quot(D)$ if and only if $|Z\setminus Z_0|=|\phi^{-1}(\overline{X}\setminus X)|=1$.

However, if $\nu':Y'\longrightarrow Z$ is a normalization, then $\phi\circ\nu_0:Y'\longrightarrow\overline{X}$ is a normalization; without loss of generality, we can suppose that $Y'=Y$ and $\phi\circ\nu_0=\nu$. Since every point of $Z\setminus Z_0$ is regular, $\nu_0$ is injective on $\nu_0^{-1}(Z\setminus Z_0)$; hence, $|\nu^{-1}(\overline{X}\setminus X)|=|Z\setminus Z_0|=|\phi^{-1}(\overline{X}\setminus X)|$. The claim follows.

\bigskip

We now prove the last statement. By the first part of the proof, $\recip(D)\subseteq\local_{\overline{X},p}$ (which is a DVR); we claim that $\recip(D)$ is not contained in any other valuation ring of dimension $1$.

Indeed, let $\nu:Y\longrightarrow\overline{X}$ be a normalization of $\overline{X}$; since $Y$ is regular, its points correspond bijectively to the one-dimensional discrete valuation rings of $k(Y)$ \cite[Corollary 4 to Theorem 1]{fulton}. In particular, if $q\in X$ then the valuation domains dominating $q$ are the local rings of the points of $\nu^{-1}(q)$. Let $W$ be any such ring, and let $f\in k[X]$ be a regular function such that $f(q)=0$. Then, $1/f\in\recip(D)$, but $f\in\mm_{\local_{X,q}}\subseteq \mm_W$ and thus $1/f\notin W$. Hence $\recip(D)$ is not contained in any other discrete valuation ring. Since $k(Y)$ has transcendence degree $1$ over $k$, it follows that the integral closure of $\recip(D)$ is $\local_{\overline{X},p}$.
\end{proof}

Theorem \ref{teor:main} provides a very actionable way to determine whether a one-dimensional algebra is Egyptian. We give a few examples.
\begin{ex}
Let $k$ be an algebraically closed field of characteristic $p$, and let $n>1$ be an integer such that $p\not|n$. Then, $X=V(x^n+y^n-1)\subset\aff^2_k$ is an affine curve with projective closure $\overline{X}=V(x_1^n+x_2^n-x_0^n)\subseteq\proj^2_k$. Since $\overline{X}$ is nonsingular, $X$ is a realization of $D=k[x,y]/(x^n+y^n-1)$ that is regular at infinity. Since $\overline{X}\setminus X$ has $n$ points (namely, $[0:1:\zeta]$ with $\zeta^n=1$), $X$ is Egyptian and $\recip(D)=\quot(D)$.
\end{ex}

\begin{ex}
Let $k$ be an algebraically closed field of characteristic $p\neq 2,3$, and let $X=V(y^2-x^3-ax-b)$ be a (possibly singular) elliptic curve in $\aff^2_k$. Then, $X$ is a realization of $D:=k[x,y]/(y^2-x^3-ax-b)$ that is regular at infinity. Since $\overline{X}\setminus X=\{[0:1:0]\}$ is a single point, $X$ is not Egyptian, and the reciprocal complement $\recip(D)$ of $D$ is not a field.
\end{ex}

\begin{ex}
Let $k$ be an algebraically closed field, and let $X=V(y-x^3)\subset\aff^2_k$. Then, $\overline{X}\setminus X=\{[0:0:1]\}$ is a single point that is not regular. Hence, $X$ is a realization of $k[X]\simeq k[t]$ (where $t$ is an indeterminate over $k$) that is not regular at infinity, but $\recip(D)=k[t^{-1}]_{(t^{-1})}$ is not a field.
\end{ex}

\begin{ex}
Let $k$ be an algebraically closed field, and let $X=V(x^3-xy-y)\subset\aff^2_k$. Then, $\overline{X}\setminus X=\{[0:0:1]\}$ is a single point; however, $[0:0:1]$ is a singular point, and there are two points of the normalization of $\overline{X}$ over $[0:0:1]$ (corresponding to the two branches). Therefore, if $D=K[x,y]/(x^3-xy-y)$, we have $\recip(D)=\quot(D)$.
\end{ex}

We now want to gauge how far is $\recip(D)$ from being a discrete valuation ring. Suppose that $k$ has characteristic $0$, and let $X$ be a nonsingular projective curve over $k$. The \emph{Weierstrass semigroup} of $X$ at $p$ (which we denote by $H(X,p)$) is the set of all integers $s$ such that there is a rational function $f$ on $X$ whose pole divisor is exactly $s\cdot p$. Then, $H(X,p)$ is a numerical semigroup (i.e., a submonoid $S$ of $\insN$ such that $\insN\setminus S$ is finite) and the cardinality of $\insN\setminus H(X,p)$ is equal to the genus of $X$; if $H(X,p)\neq\{0,g+1,g+2,\ldots\}$, then $p$ is said to be a \emph{Weierstrass point} of $X$. Every nonsingular projective curve $X$ has only finitely many Weierstrass points. See for example \cite[Chapter 11]{kazaryan} or \cite[Chapter 2, \textsection 4]{griffiths-harris}.

A \emph{Dedekind domain} is a Noetherian domain that is locally a DVR, or equivalently a one-dimensional regular ring that is an integral domain.
\begin{prop}\label{prop:weierstrass}
Let $k$ be an algebraically closed field of characteristic $0$, and let $D$ be a Dedekind domain that is finitely generated over $k$. Suppose that $\recip(D)$ is not a field, and let $X$ be a realization of $D$ that is regular at infinity. Let $\{p\}=\overline{X}\setminus X$. Then:
\begin{enumerate}[(a)]
\item\label{prop:weierstrass:valut} if $\mathbf{v}$ is the valuation on $\local_{\overline{X},p}$ and $\mu:=\min H(X,p)\setminus\{0\}$, then
\begin{equation*}
H(X,p)\subseteq\mathbf{v}(\recip(D))\subseteq \{0,\mu,\mu+1,\ldots\};
\end{equation*}
\item\label{prop:weierstrass:Wpoint} if $p$ is not a Weierstrass point of $X$, then $\mathbf{v}(\recip(D))=H(X,p)$;
\item\label{prop:weierstrass:dvr} $\recip(D)$ is a DVR if and only if $X$ has genus $0$.
\end{enumerate}
\end{prop}
\begin{proof}
We first note that, since $D$ is a Dedekind domain, $X$ is regular, and thus $\overline{X}$ is a regular projective curve. Moreover, $\overline{X}\setminus X=\{p\}$ is a point by Theorem \ref{teor:main}.

\ref{prop:weierstrass:valut} If $s\in H(X,p)$, let $f\in K(X)$ be a rational function whose pole divisor is exactly $s\cdot p$. Then, $f\in k[X]$, and $\mathbf{v}(1/f)$ is the order of the zero of $1/f$ at $p$, i.e., $s$; thus $s=\mathbf{v}(1/f)\in\mathbf{v}(\recip(D))$.

On the other hand, suppose $s\in\mathbf{v}(\recip(D))$, and let $f\in\recip(D)$ be such that $\mathbf{v}(f)=s$. Then, there are $g_1,\ldots,g_t\in k[X]\setminus\{0\}$ such that $f=\inv{g_1}+\cdots+\inv{g_t}$; since the constants are in $\recip(D)$, we can suppose without loss of generality that no $g_i$ is constant. Then, $\mathbf{v}\left(\inv{g_i}\right)\geq\mu$ for every $i$, and thus also $\mathbf{v}(f)\geq\mu$.

\ref{prop:weierstrass:Wpoint} follows directly from \ref{prop:weierstrass:valut} and the definition of Weierstrass point.

\ref{prop:weierstrass:dvr} If $\recip(D)$ is a DVR, then $\mathbf{v}(\recip(D))=\insN$, and thus we must have $\mu=1$, i.e., $H(X,p)=\insN$ and $X$ has genus $0$. Suppose $\recip(D)$ is not a DVR: since $\recip(D)$ contains $k$ and the composition $k\longrightarrow\recip(D)\longrightarrow\recip(D)/\mm_{\recip(D)}=k$ is surjective, by \cite[Proposition 1]{matsuoka_degree} the length of $\local_{\overline{X},p}/\recip(D)$ is equal to $|\insN\setminus\mathbf{v}(\recip(D))|>0$. Hence $\mathbf{v}(\recip(D))\neq\insN$, and thus $H(X,p)\neq\insN$. Therefore, the genus of $X$ is positive.
\end{proof}

Theorem 6.1 of \cite{guerrieri-recip} states that, if $D$ is a domain and $\recip(D)$ is a one-dimensional DVR, then $D$ must be isomorphic to the polynomial ring $k[t]$ for some field $k$. The last point of Proposition \ref{prop:weierstrass} shows that, in the context of one-dimensional regular $k$-algebras (when $k$ is of characteristic $0$), this result is a reformulation of the well-known fact that every nonsingular projective curve of genus $0$ is isomorphic to the projective line.

\bibliographystyle{plain}
\bibliography{/bib/articoli,/bib/libri,/bib/miei}
\end{document}